\documentclass[11pt,reqno]{amsart}

\usepackage{amsmath, amsfonts, amsthm, amssymb}

\numberwithin{equation}{section}

\newtheorem{Theorem}{Theorem}[section]
\newtheorem{Lemma}{Lemma}[section]

\theoremstyle{definition}

\theoremstyle{remark}

\renewcommand{\r}{\rho}

\renewcommand{\u}{{\bf u}}
\renewcommand{\H}{{\bf H}}
\newcommand{\R}{{\mathbb R}}

\newcommand{\na}{\nabla}
\newcommand{\dl}{\delta}

\def\f{\frac}

\def\ov{\overline}

\def\hf1{^\f{1}{1-\xi^2}}

\def\be{\begin{equation}}
\def\en{\end{equation}}
\def\bs{\begin{split}}
\def\es{\end{split}}

\def\XXint#1#2#3{{\setbox0=\hbox{$#1{#2#3}{\int}$}
    \vcenter{\hbox{$#2#3$}}\kern-.5\wd0}}

\title[Local Solution of  Magnetohydrodynamic Equations]
{Local Strong Solution to the Compressible Magnetohydrodynamic
Flow with Large Data}

\author{Xiaoli li,  Ning Su,   and Dehua Wang}

\address{Department of Mathematical Sciences,  Tsinghua University,
                           Beijing 100084, China;  and
                           Department of Mathematics, University of Pittsburgh,
                           Pittsburgh, PA 15260, USA.}
\email{xllithu@gmail.com}

\address{Department of Mathematical Sciences,  Tsinghua University,
                           Beijing 100084, China.}
\email{nsu@math.tsinghua.edu.cn}

\address{Department of Mathematics, University of Pittsburgh,
                           Pittsburgh, PA 15260, USA.}
\email{dwang@math.pitt.edu}


\date{January 5, 2011}

\begin{document}

\begin{abstract}
The three-dimensional compressible magnetohydrodynamic (MHD)
isentropic flow with zero magnetic diffusivity is studied. The vanishing magnetic diffusivity causes significant difficulties due to the loss of dissipation of the magnetic field.  The
existence and uniqueness of local in time strong solution with large
initial data is established. The strong solution has weaker
regularity than the classical solution.  A generalized Lax-Milgram
theorem and a Schauder-Tychonoff-type fixed point argument are
applied with novel techniques and estimates for the strong solution.
\end{abstract}

\maketitle

\section{Introduction}  \label{sec1}

Magnetohydrodynamics (MHD) concerns the motion of conducting fluids,
such as gases, in an electromagnetic field. If a conducting fluid
moves in a magnetic field, electric fields are induced and an
electric current flow is developed. The magnetic field exerts forces
on these currents which considerably modify the hydrodynamic motion
of the fluid. On the other hand, the development of electric
currents yields a change in the magnetic field. There is a complex
interaction between the magnetic field and fluid dynamic phenomena,
and both hydrodynamic and electrodynamic effects have to be
considered. The equations for compressible magnetohydrodynamics
consist of the Euler equations of gas dynamics coupled with the
Maxwell's equations of electromagnetic field. The applications of
magnetohydrodynamics cover a very wide range of physical areas from
liquid metals to cosmic plasmas, for example, the intensely heated
and ionized fluids in an electromagnetic field in astrophysics,
geophysics, high-speed aerodynamics and plasma physics.

The equations of three-dimensional compressible magnetohydrodynamic
flow in the isentropic case have the following form (\cite{Ca, KL,
LL}):
\begin{subequations} \label{e1}
\begin{align}
&\r_t +\nabla\cdot(\r\u)=0, \label{e11}\\
&(\r\u)_t+\nabla\cdot\left(\r\u\otimes\u\right)+\nabla P
  =(\na \times \H)\times \H+\mu \Delta \u+(\lambda+\mu)\nabla(\nabla\cdot\u), \label{e12}\\
&\H_t-\nabla\times(\u\times\H)=0,\quad \nabla\cdot\H=0, \label{e13}
\end{align}
\end{subequations}
where $\r=\r(x,t)\in\R^+$ denotes the density,
$\u=\u(x,t)
\in\R^3$ the velocity field,
$\H=\H(x,t)
\in\R^3$ the magnetic field, and
$P(\r)=A\r^\gamma$ the pressure with a constant $A>0$ and the adiabatic exponent
$\gamma>1$. The viscosity coefficients $\mu$ and $\lambda$ of the flow are constants
satisfying $\mu>0$ and $2\mu+3\lambda>0$, which ensures that the operator $-\mu \Delta
\u-(\lambda+\mu)\nabla(\nabla\cdot\u)$ is a strictly elliptic
operator. The symbol $\otimes$
denotes the usual Kronecker tensor product. Usually, we refer to
\eqref{e11} as the continuity equation\ (mass conservation
equation), and \eqref{e12} as the momentum conservation equation.
It is well-known that the electromagnetic fields
are governed by the Maxwell's equations. In magnetohydrodynamics,
the displacement current  can be neglected (\cite{KL, LL}). As a
consequence, \eqref{e13} is called the induction equation. As for
the constraint $\nabla\cdot\H=0$, it can be seen just as a
restriction on the initial value of $\H$ since $(\nabla\cdot\H)_t=0$.
We remark that, the magnetic diffusivity in \eqref{e1} is zero, which arises in the physics regime with negligible electrical resistance, see \cite{CP}.

We consider the Cauchy problem of \eqref{e1} with the initial
condition:
\begin{equation}\label{IC}
(\r, \u, \H)(x,0)=(\r_0, \u_0, \H_0)(x), \quad x\in\R^3,
\end{equation}
and are interested in the existence of solutions to
\eqref{e1}-\eqref{IC}. When the magnetic diffusivity $\nu\ne 0$,
\eqref{e13} is
$$\H_t-\nabla\times(\u\times\H)=-\nabla\times(\nu\nabla\times\H),\quad
\nabla\cdot\H=0,$$
 and  there have been many studies and rich results in the literature:
see \cite{gw, gw2, f3,FJN, fy,h1, xw,HT, LL, w1} and the references therein.
When the magnetic diffusivity $\nu= 0$ as in \eqref{e1}, the mathematical analysis becomes
 much more difficult due to the loss of dissipation of the magnetic field, and to our best
  knowledge there have been no results on existence of solutions (even in the incompressible case).
The aim of  this paper is to establish the local existence and
uniqueness of strong solution to   system \eqref{e1} with large
initial data in the three-dimensional space $\R^3$. By a strong
solution, we mean a triplet $(\r, \u, \H)$ with $\u(\cdot, t)\in
W^{2,q}(\R^3)$ and $(\r(\cdot, t), \H(\cdot, t))\in W^{1,q}(\R^3),
\, 3<q\leq6$ satisfying \eqref{e1} almost everywhere with the
initial condition \eqref{IC}.
As for the global existence of classical
solutions of the small perturbation near an equilibrium for
compressible Navier-Stokes equations, we refer the interested
reader to \cite{MT1, MT} and the references cited therein. The
global existence of strong solutions with small perturbations near
an equilibrium for compressible Navier-Stokes equations was also
discussed in \cite{AI, SS}. Also see the discussions and references in \cite{AI, SS} for other related results on strong solutions.

Throughout this paper, the standard notations for Sobolev spaces
$W^{s, p}(\R^3)$ ($H^s(\R^3)$, when $p=2$) will be used. For
$p\in[1,+\infty]$, we denote by $L^p(0,T; X)$ the set of Bochner
measurable $X-$valued time dependent functions $\varphi$ such that
$t\mapsto \|\varphi\|_X$ belongs to $L^p(0,T),$ and the
corresponding Lebesgue norm is denoted by $\|\cdot\|_{L^p_T(X)}$.
Denote the Sobolev space
$W^{1,p}(0,T;X):=\{\varphi\mid \varphi\in
L^p(0,T;X), \  \varphi_t\in L^p(0,T;X)\}$,  and
$H^1(0,T; X):=W^{1,2}(0,T;X)$.
Precisely, we will establish the following result on existence and uniqueness in
this paper:
\begin{Theorem}\label{T20}
Assume that  $$\r_0\in W^{1,q}(\R^3)\cap H^1(\R^3),\ \ \u_0\in
\big(H^2(\R^3)\big)^3,\  \ \H_0\in \big(W^{1,q}(\R^3)\big)^3\cap
\big(H^1(\R^3)\big)^3$$ for some $q\in(3, 6]$,  and
$$\alpha\le \r_0\le\beta, \ \ \|\r_0\|_{W^{1,q}\cap
H^1}+\|\u_0\|_{H^2}+\|\H_0\|_{W^{1,q}\cap H^1} \le r_0$$
for some positive constants $\alpha, \beta$, and $r_0$.
Then there exist  positive constants  $\overline{T}=\ov{T}(r_0)$, $\alpha_1 (\ov{T}, r_0, \alpha)$,
and $\beta_1(\ov{T}, r_0, \beta)$, such that the Cauchy problem
\eqref{e1}-\eqref{IC} has a unique strong  solution $(\r,\u,\H)$ on
$\R^3\times(0,\ov{T})$ satisfying
$$\r\in L^\infty\big(0,\ov{T}; \ W^{1,q}(\R^3)\cap H^1(\R^3)\big);\quad \r_t\in L^\infty\big(0,\ov{T};
L^q(\R^3)\big);$$
$$\alpha_1(\ov{T},r_0,
\alpha)\le\r\le\beta_1(\ov{T},r_0,\beta);$$
$$\u\in \big(L^2(0,\ov{T}; \ W^{2,q}(\R^3)\cap H^2(\R^3))\big)^3; \quad \u_t\in
\big(L^2(0,\ov{T}; H^1(\R^3))\big)^3;$$
$$\H\in \big(L^\infty(0,\ov{T};\  W^{1,q}(\R^3)\cap H^1(\R^3))\big)^3; \quad \H_t\in \big(L^\infty\big(0,\ov{T};
L^q(\R^3))\big)^3.$$

\end{Theorem}

In addition to the difficulties due to the presence of the magnetic
field and its interaction with the hydrodynamic motion in the MHD
flow of large oscillation, another major difficulty in proving the
existence is the lack of the dissipative estimates for the magnetic
field and the gradient of the density. Since we are concerned with
the {\em strong solutions in $W^{2,q}$} which have weaker regularity
than the {\em classical solution in $H^3$}, we need some new
techniques and estimates to establish the existence. We first
linearize \eqref{e1},   use the Lax-Milgram theorem to obtain the
solution to the linearized system, then
 apply the Schauder-Tychonoff fixed point theorem to obtain the
strong solution of \eqref{e1} with large data.

 This paper is organized as follows. Section 2, which is the main body of this
paper, is devoted to proving the local existence of the system
\eqref{e1} by Lax-Milgram theorem and a fixed-point argument.
Section 3 will focus on the uniqueness of the solution obtained in
Section 2.

\bigskip\bigskip
\section{Local Existence}

We use the letter $C$ to denote any
constant that can be explicitly computed in terms of known
quantities, and the exact value denoted by $C$ may therefore change
from line to line in a given computation. Similarly, $\varepsilon$ and $
\delta$ denote arbitrary positive constants, $C_{\varepsilon}$ and
$C_{\delta}$ denote correspondingly positive constants depending on
$\frac{1}{\varepsilon}$ and $\frac{1}{\delta}$.
 Set $\partial_i=\partial/{\partial x_i},\ i=1,2,3$ for $x=(x_1, x_2, x_3)\in\R^3$ .
If $A=((a_{ij}))\ \textrm{and}\  B=((b_{ij}))\  \textrm{are}\
3\times 3\ \textrm{matrices}$, then
$$A:B=\sum_{i=1}^3\sum_{j=1}^3 a_{ij}b_{ij}
\ \ \textrm{and} \ \
|A|=(A:A)^{\frac{1}{2}}=\big(\sum_{i=1}^3\sum_{j=1}^3
a_{ij}^2\big)^{\frac{1}{2}}.$$
$$|\u|_p=\textrm{norm in the space}\ L^p(\R^3)\
\textrm{with}\  p\geq1;$$
$$ \|\u\|_s=\textrm{norm in the Sobolev
space}\  H^s(\R^3)\ \textrm{of order} \ s \ \textrm{on} \
L^2(\R^3);$$
$$\|\u\|_{s,p}=\textrm{norm in the Sobolev space}\ W^{s,p}(\R^3)\
\textrm{of order}\  s \ \textrm{on} \ L^p(\R^3).$$
For a given $T>0$,  denote
$$Q_T:=\R^3\times[0,T].$$
Obviously,\ $L^p(0,T; L^p(\R^3))=L^p(Q_T)$.

In this section, we will prove the existence result in Theorem
\ref{T20}. For simplicity of notations, we set
$\mu=\beta=1$ and $\lambda=0$ without loss of generality. The proof
will proceed  through four steps by combining a generalized
Lax-Milgram theorem and a Schauder fixed point argument. To this
end, we consider first an auxiliary problem.

Set
\begin{equation*}
\begin{split}
\Phi=\big\{\phi(x,t)\in\R^3\mid\phi\in \big(L^2(0,T;
H^2(\R^3))\big)^3,\ \phi_t\in \big(L^2(Q_T)\big)^3\},
\end{split}
\end{equation*}
with the natural norm\ $\|{\bf\phi}\|_{\Phi}$.  And, for $q\in
(3,6]$,\ \ define
\begin{equation*}
\begin{split}
\Psi=\Phi&\cap \Big\{\phi(x,t)\in\R^3|\ \phi\in \big(L^2(0,T;
W^{2,q}(\R^3))\big)^3\cap \big(L^\infty(0,T;
H^2(\R^3))\big)^3\Big\}\\& \cap \Big\{\phi(x,t)\in\R^3|\ \phi\in
\big(L^2(0,T; H^2(\R^3))\big)^3, \ \phi_t \in
\big(L^\infty(0,T;L^2(\R^3))\big)^3, \\
&\qquad\qquad\qquad\qquad \nabla \phi_t\in \big(L^2(Q_T)\big)^9,\
\phi(0)=\u_0\Big\}.
\end{split}
\end{equation*}

Using the continuity equation \eqref{e11}, the momentum  equation
\eqref{e12} can be reduced to $$\r\u_t+\r\u\cdot\nabla \u+\nabla
P=(\nabla\times \H)\times \H+\triangle\u+\nabla(\nabla \cdot\u),$$
where the notation $\u\cdot\nabla\u$ is understood to be
$(\u\cdot\nabla)\u$. Meanwhile, under the constraint
$\nabla\cdot\H=0$, the induction equation is equivalent to
$$\H_t+{\bf{u}}\cdot\nabla\H=\big(\nabla{\bf{u}}-(\nabla\cdot{\bf{u}}){\bf
I}\big)\ \H,$$ where ${\bf I}$ stands for the $3\times3$ identity
matrix.

 We need to find a triplet $(\r,\u,\H)$ that satisfies the following auxiliary problem
\begin{subequations}\label{e3}
\begin{align}
&\r_t+\nabla \cdot(\r{\bf\bar{u}})=0, \label{e31}\\
&\r\u_t-\triangle\u-\nabla(\nabla \cdot\u)=-\r{\bf\bar{u}}\cdot\nabla {\bf \bar{u}}-\nabla P+(\nabla\times \H)\times \H,  \label{e32}\\
&\H_t+{\bf\bar{u}}\cdot\nabla\H=\big(\nabla{\bf\bar{u}}-(\nabla\cdot{\bf\bar{u}}){\bf
I}\big)\ \H, \ \nabla \cdot \H=0, \label{e33}
\end{align}
\end{subequations}
a.e. in $Q_T$ for any given $T>0$, with the initial condition \eqref{IC} such that
$\u\in\Psi,$
$$\r\in L^\infty\big(0,T;W^{1,q}(\R^3)\cap
H^1(\R^3)\big)\cap H^1\big(0,T;L^q(\R^3)\cap L^2(\R^3)\big),$$
$$\alpha_1\leq\r\leq\beta_1\ (\alpha_1,\beta_1\ \textrm{are suitable positive constants}),$$
and
$$\H\in \big(L^\infty(0,T;W^{1,q}(\R^3)\cap H^1(\R^3))\big)^3\cap
\big(H^1(0,T;L^q(\R^3)\cap L^2(\R^3))\big)^3.$$
where ${\bf\bar{u}}\in\Psi$,  and $\u_0\in
\big(H^2(\R^3)\big)^3$, $\r_0\in W^{1,q}(\R^3)\cap H^1(\R^3)$ with
$\alpha\leq\rho_0\leq 1$, $\H_0\in \big(W^{1,q}(\R^3)\big)^3\cap
\big(H^1(\R^3)\big)^3$ are given functions and $q\in(3, 6]$.

\subsection{Solvability of the density with a fixed velocity.}
Obviously,  the existence of a unique solution $\r:=\r(\bar{\u})$ of
the continuity equation  follows directly from the method of
characteristics. Although this method requires that ${\bf\bar{u}}\in
\big(C^0(0,T;C^1(\R^3))\big)^3$ and $\rho_0\in C^1(\R^3)$, the
estimates below in Lemma \ref{r} hold under the above assumptions on
${\bf\bar{u}}$ and $\rho_0$, i.e., ${\bf\bar{u}}\in\Psi,\ \r_0\in
W^{1,q}(\R^3)\cap H^1(\R^3)$. So we use, for simplicity, a formal
approach. The correct procedure is to consider a regularization of
${\bf\bar{u}}$ and $\rho_0$, and then to pass to limit (similar to
the argument for $\H$ below, and can also be found in Theorem 9.3 in
\cite {AI}).
\begin{Lemma}\label{r}
Under the same conditions as Theorem \ref{T20}, there is a unique
strictly positive function
$$\r:=\r(\bar{\u})\in H^1\big(0,T; L^q(\R^3)\cap L^2(\R^3)\big)\cap L^\infty\big(0,T; W^{1,q}(\R^3)\cap H^1(\R^3)\big)$$
which satisfies \eqref{e31}. Moreover, the density satisfies the
following estimate:
\begin{equation}\label{r1}
\|\nabla \r\|_{L^\infty(0,T; L^q(\R^3)\cap L^2(\R^3))}\le
\left(\|\r_0\|_{W^{1,q}(\R^3)\cap H^1(\R^3)}
+C\|\bar{\u}\|_\Psi\sqrt{T}\right)\exp\left(C\|\bar{\u}\|_\Psi\sqrt{T}\right).
\end{equation}
\end{Lemma}

\begin{proof}
Along characteristics
\begin{equation*}
\begin{cases}
\frac{dX}{dt}={\bf\bar{u}}(t,X),\\ X(t)=x,
\end{cases}
\end{equation*}
 \eqref{e31} can be
rewritten as
$$\frac{d}{dt}\r(t,X(t))=-\r(t,X(t))\nabla\cdot{\bf\bar{u}}(t,X(t)).$$
Then, the explicit formula for $\r$ is
\begin{equation*}
\r(t,x)=\r_0(X(0))\textrm{exp}\left(-\int_0^t\nabla\cdot{\bf\bar{u}}(\tau,X(\tau))\
d\tau\right).
\end{equation*}
 It follows that
\begin{equation}\label{2}
\alpha\,
\textrm{exp}\left(-\int_0^t|\nabla\cdot{\bf\bar{u}}(\tau,X(\tau))|_\infty
d\tau\right)\leq\r(t,x)\leq
\textrm{exp}\left(\int_0^t|\nabla\cdot{\bf\bar{u}}(\tau,X(\tau))|_\infty
d\tau\right).
\end{equation}

 Now applying the gradient operator $\nabla$ to
\eqref{e31}, and using
\begin{equation*}
\nabla(\nabla\r\cdot\bar{\u})
=\bar{\u}\cdot\nabla(\nabla\r)+\nabla\bar{\u}\cdot\nabla\r,
\end{equation*}
where $\nabla\bar{\u}\cdot\nabla\r=(\nabla\bar{\u})^\top\nabla\r$,
 we have
\begin{equation}\label{22}
\nabla\r_t+{\bf\bar{u}}\cdot\nabla (\nabla\rho)+\nabla
{\bf\bar{u}}\cdot\nabla\r+\r\nabla(\nabla\cdot{\bf\bar{u}})+\nabla\cdot{\bf\bar{u}}\nabla\r=0.
\end{equation}
Multiplying \eqref{22} by $|\nabla\rho|^{q-2}\nabla\rho$ with
$3<q\leq 6$ and integrating over $\R^3$, we get
\begin{equation*}
\begin{split}
&\frac{1}{q}\frac{d}{dt}|\nabla\r|_q^q+\int_{\R^3}\Big(\frac{1}{q}{\bf\bar{u}}\cdot\nabla(|\nabla\r|^q)
+|\nabla\r|^{q-2}\nabla\r\cdot(\nabla{\bf\bar{u}}\cdot\nabla\r)\\
&\quad\quad\quad\qquad\qquad+\r|\nabla\r|^{q-2}\nabla\r\cdot\nabla(\nabla\cdot{\bf\bar{u}})
+|\nabla\r|^q\nabla\cdot{\bf\bar{u}}\Big)\ dx=0.
\end{split}
\end{equation*}
Bearing in mind, by divergence theorem,
\begin{equation*}
\begin{split}
\int_{\R^3}({\bf\bar{u}}\cdot\nabla(|\nabla\r|^q)+|\nabla\r|^q
\nabla\cdot{\bf\bar{u}})\
dx&=\int_{\R^3}\nabla\cdot(|\nabla\r|^q{\bf\bar{u}})\ dx=0,
\end{split}
\end{equation*}
we find, by H\"{o}der's inequality,
\begin{equation*}
\begin{split}
\frac{1}{q}\frac{d}{dt}|\nabla\r|_q^q&\leq
|\nabla\r|_q^{q-1}(|\r|_\infty|\nabla(\nabla\cdot{\bf\bar{u}})|_q+|\nabla\r|_q|\nabla{\bf\bar{u}}|_\infty)+
(\frac{1}{q}-1)\int_{\R^3}|\nabla\r|^q\nabla\cdot{\bf\bar{u}}\ dx
\\&\leq |\nabla\r|_q^{q-1}(|\r|_\infty|\nabla(\nabla\cdot{\bf\bar{u}})|_q+|\nabla\r|_q|\nabla{\bf\bar{u}}|_\infty)+|\nabla\r|_q^q|\nabla\cdot{\bf\bar{u}}|_\infty \\
&\leq C|\nabla\r|^q_q|\nabla
{\bf\bar{u}}|_\infty+|\nabla\r|_q^{q-1}|\r|_\infty|\nabla(\nabla\cdot{\bf\bar{u}})|_q.
\end{split}
\end{equation*}
Since
\begin{equation}\label{im}
W^{1,q}(\R^3)\hookrightarrow L^\infty(\R^3)\quad \ \textrm{as}\quad
3<q\leq 6,
\end{equation}
then we have
\begin{equation*}
\begin{split}
\frac{d}{dt}|\nabla\r|_q&\leq C|\nabla\r|_q|\nabla
{\bf\bar{u}}|_\infty+|\r|_\infty|\nabla(\nabla\cdot{\bf\bar{u}})|_q\\
&\leq
C\|\bar{\u}\|_{2,q}|\nabla\r|_q+|\r|_\infty|\nabla(\nabla\cdot{\bf\bar{u}})|_q,
\end{split}
\end{equation*}
and, by Gronwall's inequality,
\begin{equation}\label{20}
|\nabla\r|_q\leq \textrm{exp}\left(C\int_0^t\|{\bf\bar{u}}\|_{2,q}
d\tau)\big(|\nabla \r_0|_q+\int_0^t|\r|_\infty|\nabla(\nabla
\cdot{\bf\bar{u}})|_q d\tau\right).
\end{equation}
Finally, \eqref{r1} follows from \eqref{20} and H\"{o}der's inequality. The proof is complete.

\end{proof}

\subsection{Solvability of the magnetic field with a fixed velocity.}

Due to the hyperbolic structure of \eqref{e33} in terms of the
magnetic field $\H:=\H(\bar{\u})$ with the fixed velocity
$\bar{\u}$, we can solve $\H$ through the following Lemma \ref{l1}.

Let $A_j(x, t)$, $j=1, ..., n$,  be symmetric $m\times m$ matrices
in $\R^n\times(0,T)$, $f(x,t)$ and $V_0(x)$ two $m$-dimensional
vector functions defined in $\R^n\times(0,T)$ and $\R^n$,
respectively.
For the Cauchy problem of the linear system in
$V\in\R^m\times(0,T)$:
\begin{equation}\label{1000}
\begin{cases}
&\displaystyle V_t+\sum_{i=1}^nA_j(x,t)\partial_jV+B(x,t)V=f(x,t),\\
&V(x,0)=V_0(x),
\end{cases}
\end{equation}
we have
\begin{Lemma}\label{l1}
Assume that $$A_j\in \big(C(0,T; H^s(\R^n))\cap C^1(0,T;
H^{s-1}(\R^{n}))\big)^{m\times m}, \;  j=1,...,n,$$
$$B\in \big(C(0,T; H^{s-1}(\R^n))\big)^{m\times m},\quad f\in \big(C(0,T; H^s(\R^n))\big)^m,\quad V_0\in \big(H^s(\R^n)\big)^m,$$
with $s>\f{n}{2}+1$  an integer. Then there exists a unique solution
to \eqref{1000}, i.e, a function
$$V\in \big(C([0,T), H^s(\R^n))\cap C^1((0,T), H^{s-1}(\R^n))\big)^m$$
satisfying \eqref{1000} pointwise ( i.e. in the classical sense).
\end{Lemma}
\begin{proof}
This lemma is a direct consequence of Theorem 2.16 in \cite{AI} with
$A_0(x,t)={\bf I}$.
\end{proof}

Now, taking advantage of Lemma \ref{l1}, we have
\begin{Lemma}\label{l2}
Under the same conditions as Theorem \ref{T20}, there is a unique
function
$$\H:=\H(\bar{\u})\in \big(H^1(0,T; L^q(\R^3)\cap L^2(\R^3))\big)^3\cap \big(L^\infty(0,T; W^{1,q}(\R^3)\cap H^1(\R^3))\big)^3$$
which satisfies the equation \eqref{e33}. Moreover, the magnetic
field satisfies
$$\|\H\|_{L^\infty(0,T; W^{1,q}(\R^3)\cap H^1(\R^3))}\le \left(\|
\H_0\|_{W^{1,q}(\R^3)\cap H^1(\R^3)
}+C\|\bar{\u}\|_\Psi\sqrt{T}\right)\exp(C\|\bar{\u}\|_\Psi\sqrt{T}).$$
\end{Lemma}

\begin{proof}
Assume that $\bar{\u}\in \big(C^1(0,T;C_0^\infty(\R^3))\big)^3$ and
$ \H_0\in \big(C_0^\infty(\R^3)\big)^3$. Rewriting \eqref{e33} in
the component form \big($\bar{\u}=(\bar{\u}^{(1)}, \bar{\u}^{(2)},
\bar{\u}^{(3)}), \ \H=(\H^{(1)}, \H^{(2)}, \H^{(3)})$\big),
employing the summation convention on the repeated indices, we have
$$\H_t^{(i)}+\bar{\u}^{(j)}\partial_j{\H}^{(i)}=\big(\partial_j \bar{\u}^{(i)}-(\nabla\cdot{\bf\bar{u}})\delta_{ij}\big)\ \H^{(j)},\ \ i=1,2,3.$$
i.e,  $$\H_t+(\bar{\u}^{(j)}{\bf I})\
\partial_j{\H}+((\nabla\cdot{\bf\bar{u}}){\bf I}-\nabla{\bf\bar{u}})\ \H=0.$$

Clearly, $A_j(x,t)=\bar{\u}^{(j)}(x,t) {\bf I}$, $ j=1,2, 3,\
B(x,t)=(\nabla\cdot{\bf\bar{u}}){\bf I}-\nabla{\bf\bar{u}}$ \ and
$f(x,t)=0$ satisfy the assumptions in Lemma \ref{l1}, then we get a
unique solution
$$\H\in \bigcap_{s=3}^\infty\big\{\big(C^1(0,T, H^{s-1}(\R^3))\cap C(0,T; H^s(\R^3))\big)^3\big\},$$
which implies, by the Sobolev imbedding theorems,
$$\H\in \bigcap_{k=1}^\infty \big(C^1(0,T; C^k(\R^3))\big)^3=\big(C^1(0,T; C^\infty(\R^3))\big)^3.$$

 For ${\bf\bar{u}}\in \Psi$, $\H_0\in \big(W^{1,q}(\R^3)\big)^3\cap
\big(H^1(\R^3)\big)^3,$ by an argument of dense set, there are
 $\{{\bf\bar{u}}_n\}_{n=1}^\infty\subset
\big(C^1(0,T;C_0^\infty(\R^3))\big)^3$,\
$\{{\H_0}_n\}_{n=1}^\infty\subset \big(C_0^\infty(\R^3)\big)^3$
respectively such that
$${\bf\bar{u}}_n\rightarrow {\bf\bar{u}}\ \ \textrm{in}\ \Psi,$$
$${\H_0}_n\rightarrow \H_0 \ \ \textrm{in}\
\big(W^{1,q}(\R^3)\big)^3\cap \big(H^1(\R^3)\big)^3.$$
  Hence,
$${\bf\bar{u}}_n\rightarrow {\bf\bar{u}}\ \textrm{in}\  \big(C([0,T]\times
B(0,a))\big)^3,$$ for any $a>0$ and $B(0,a)$ denotes the ball with
radius $a$ and centered at the origin.

 Similarly, there are
$\{\H_n\}_{n=1}^\infty\subset\big(C^1(0,T; C^\infty(\R^3))\big)^3$
satisfying
\begin{equation}\label{21}
{\H_n}_t+{\bf\bar{u}}_n\cdot\nabla
\H_n=\big(\nabla\bar{\u}_n-(\nabla\cdot\bar{\u}_n){\bf I} \big) \H_n
\end{equation}
with $\H_n(0)={\H_0}_n$.
Multiplying \eqref{21} by $|\H_n|^{p-2}\H_n\ (p\geq 2) $,
integrating over $\R^3$, by integration by parts, we obtain,
\begin{equation*}
\begin{split}
\f{1}{p}\f{d}{dt}|\H_n|_p^p&=-\f{1}{p}\int_{\R^3}
{\bf\bar{u}}_n\cdot\nabla (|\H_n|^p) dx
+\int_{\R^3}|\H_n|^{p-2}\H_n\cdot\big(\nabla\bar{\u}_n-(\nabla\cdot\bar{\u}_n){\bf I} \big)\ \H_n dx\\
&=\f{1}{p}\int_{\R^3} |\H_n|^p\nabla \cdot\bar{\u}_n dx
+\int_{\R^3}|\H_n|^{p-2}\H_n\cdot\big(\nabla\bar{\u}_n-(\nabla\cdot\bar{\u}_n){\bf I} \big)\ \H_n dx\\
&\le C|\H_n|_p^p|\nabla{\bf\bar{u}}_n|_\infty,
\end{split}
\end{equation*}
where $\H_n\cdot\big(\nabla\bar{\u}_n-(\nabla\cdot\bar{\u}_n){\bf I}
\big)=\H_n^\top\big(\nabla\bar{\u}_n-(\nabla\cdot\bar{\u}_n){\bf I}
\big).$
Then, by Gronwall's inequality and the imbedding \eqref{im}, we get
\begin{equation*}
\begin{split}
|\H_n|_p^p \le \textrm{exp}\left(C\int_0^t|\nabla
{\bf\bar{u}}_n|_\infty
d\tau\right)|\H_n(0)|_p^p
\le\textrm{exp}\left(C\int_0^t\|
{\bf\bar{u}}_n\|_{2,q}d\tau\right)|{\H_0}_n|_p^p.
\end{split}
\end{equation*}
i.e.,
$$|\H_n|_p\le \textrm{exp}\left(C\int_0^t\|
{\bf\bar{u}}_n\|_{2,q}d\tau\right)|{\H_0}_n|_p.$$
 Thus, by H\"{o}lder's inequality, we have
\begin{equation}\label{n}
\|\H_n\|_{L^\infty_T(L^p(\R^3))}\le \textrm{exp}(C\
\|\bar{\u}_n\|_\Psi\sqrt{T})|{\H_0}_n|_p.
\end{equation}
In particular, if we choose $p=q$ in \eqref{n}, then
\begin{equation*}
\begin{split}
\|\H_n\|_{L^\infty_T(L^q(\R^3))}\le \textrm{exp}(C\
\|\bar{\u}_n\|_\Psi\sqrt{T})|{\H_0}_n|_q
<\infty,
\end{split}
\end{equation*}
 and, up to a
subsequence, assuming that $\{{\H}_n\}_{n=1}^\infty$ were chosen so
that
$$\H_n\rightarrow \H \quad \textrm{weak-* in}\quad \big(L^\infty(0,T; L^q(\R^3))\big)^3.$$
Moreover, letting $p\rightarrow\infty$ in \eqref{n}, using the
imbedding \eqref{im} again, we obtain
\begin{equation*}
\begin{split}
\|\H_n\|_{L^\infty(Q_T)}&\le
\textrm{exp}(C\|{\bf\bar{u}}_n\|_\Psi\sqrt{T})|{\H_0}_n|_\infty\\
&\le C\textrm{exp}(C\|{\bf\bar{u}}_n\|_\Psi\sqrt{T})\|{\H_0}_n\|_{1,q}<\infty.
\end{split}
\end{equation*}

Taking the gradient in both sides of \eqref{21}, multiplying by
$|\nabla \H_n|^{q-2}\nabla \H_n$ and integrating over $\R^3$, we
get, with the help of H\"{o}lder's inequality and the imbedding
\eqref{im},
\begin{equation}\label{y3}
\begin{split}
&\f{1}{q}\f{d}{dt}|\nabla
\H_n|^q_q\\&=-\int_{\R^3}|\nabla\H_n|^{q-2}\nabla\H_n:\big(\nabla(\bar{\u}_n\cdot\nabla\H_n)\big)dx\\
&\quad+\int_{\R^3}|\nabla\H_n|^{q-2}\nabla\H_n:\big(\nabla\big((\nabla\bar{\u}_n-(\nabla\cdot\bar{\u}_n){\bf
I})\H_n\big)\big)dx\\
&=-\int_{\R^3}|\nabla\H_n|^{q-2}\sum_{i=1}^3\big(\partial_i\H_n\cdot\partial_i(\bar{\u}_n\cdot\nabla\H_n)
-\partial_i\H_n\cdot\partial_i\big((\nabla\bar{\u}_n-(\nabla\cdot\bar{\u}_n){\bf
I})\H_n\big)\big)dx\\
&=-\int_{\R^3}|\nabla\H_n|^{q-2}\sum_{i,j=1}^3\big(\partial_i\H_n^{(j)}\big(\partial_i\bar{\u}_n\cdot\nabla\H_n^{(j)}
+\bar{\u}_n\cdot\partial_i(\nabla\H_n^{(j)})\big)\big)dx\\
&\quad+\int_{\R^3}|\nabla\H_n|^{q-2}\sum_{i,j,k=1}^3\big(\partial_i\H_n^{(j)}\partial_i\H_n^{(k)}\big(\partial_k\bar{\u}_n^{(j)}
-(\nabla\cdot\bar{\u}_n)\delta_{jk}\big)\\
&\qquad\qquad+\partial_i\H_n^{(j)}\H_n^{(k)}\partial_i\big(\partial_k\bar{\u}_n^{(j)}-(\nabla\cdot\bar{\u}_n)\delta_{jk}\big)\big)dx\\
&\le
C\int_{\R^3}|\nabla\H_n|^q|\nabla\bar{\u}_n|dx-\frac{1}{q}\int_{\R^3}\bar{\u}_n\cdot\nabla(|\nabla\H_n|^q)dx+C\int_{\R^3}|\nabla
\H_n|^q|\nabla{\bf\bar{u}}_n|dx\\
&\quad
+C\int_{\R^3} |\H_n||\nabla \H_n|^{q-1}|\nabla\nabla {\bf\bar{u}}_n|dx\\
&\leq C\int_{\R^3}|\nabla\H_n|^q|\nabla\bar{\u}_n|dx+C\int_{\R^3} |\H_n||\nabla \H_n|^{q-1}|\nabla\nabla {\bf\bar{u}}_n|dx\\
&\le C|\nabla\bar{\u}_n|_\infty|\nabla\H_n|_q^q+C|\H_n|_\infty\|{\bf\bar{u}}_n\|_{2,q}|\nabla \H_n|_q^{q-1}\\
&\le C\|{\bf\bar{u}}_n\|_{2,q}|\nabla
\H_n|^q_q+C\|{\bf\bar{u}}_n\|_{2,q}|\nabla\H_n|_q^{q-1},
\end{split}
\end{equation}
Using Gronwall's inequality, we conclude that
\begin{equation*}
\begin{split}
|\nabla \H_n|_q&\le
\textrm{exp}(C\int_0^t\|{\bf\bar{u}}_n\|_{2,q}d\tau)\left(|\nabla
\H_n(0)|_q+C\int_0^t\|{\bf\bar{u}}_n\|_{2,q}d\tau\right),
\end{split}
\end{equation*}
and hence,
\begin{equation}\label{y31}
\begin{split}
|\nabla \H|_q&\le\liminf_{n\rightarrow\infty}|\nabla \H_n|_q\le
\textrm{exp}(C\int_0^t\|{\bf\bar{u}}\|_{2,q}d\tau)\left(|\nabla
\H_0|_q+C\int_0^t\|{\bf\bar{u}}\|_{2,q}d\tau\right).
\end{split}
\end{equation}
Furthermore,
$$\|\H\|_{L^\infty_T( W^{1,q}(\R^3))}\le \textrm{exp}(C\|{\bf\bar{u}}\|_\Psi\sqrt{T})\left(\|
\H_0\|_{1,q}+C\|{\bf\bar{u}}\|_\Psi\sqrt{T}\right).$$

Passing to the limit as $n\rightarrow\infty$ in \eqref{21}, we show
that \eqref{e33} holds at least in the sense of distributions.
Therefore, $\H_t\in \big(L^2(0,T; L^2(\R^3))\big)^3$, then $\H\in
\big(H^1(0,T; L^q(\R^3)\cap L^2(\R^3))\big)^3$. The proof is
complete.
\end{proof}

\subsection{Local solvability of \eqref{e32}.}
Now we prove the existence of a solution of \eqref{e32}. First we
consider the bilinear form $E(\u,{\bf\phi})$ and the functional
$L(\phi)$ defined by
\begin{equation*}
\begin{split}
E(\u,{\bf\phi})&=\int_0^T\big(\r\u_t-\triangle\u-\nabla(\nabla\cdot\u),\
\phi_t-k\big(\triangle\phi+\nabla(\nabla\cdot\phi)\big)\big)dt\\
&\quad -\big(\u(0),
\triangle\phi(0)+\nabla\big(\nabla\cdot\phi(0)\big)\big),
\end{split}
\end{equation*}
\begin{equation*}
\begin{split}
L(\phi)&=-\int_0^T\big(\r\bar{\u}\cdot\nabla\bar{\u}+\nabla
P-(\nabla\times\H)\times\H,\
\phi_t-k\big(\triangle\phi+\nabla(\nabla\cdot\phi)\big)\big)dt\\&\quad-\big(\u_0,
\triangle\phi(0)+\nabla\big(\nabla\cdot\phi(0)\big)\big)
\end{split}
\end{equation*}
with $k=(\|\r\|_{L^\infty(Q_T)})^{-1}$ for $\phi\in \Phi$. Here, and
in what follows, $(\cdot,\cdot)$ denotes the inner product in
$\big(L^2(\R^3)\big)^3$.
Obviously, $L(\phi)$ is linear continuous on $\Phi$ with respect to
the norm $\|\phi\|_\Phi$. Moreover, by the Cauchy-Schwarz inequality, we get
\begin{equation*}
\begin{split}
E(\phi,\phi)&=\int_0^T\left(|\sqrt{\r}\phi_t|^2_2+k|\triangle\phi+\nabla(\nabla\cdot\phi)|^2_2-k(\r\phi_t,
\triangle\phi+\nabla(\nabla\cdot\phi)\right)dt\\
&\quad+\f{1}{2}\left(|\nabla\phi(T)|^2_2+|\nabla\phi(0)|^2_2+|\nabla\cdot\phi(T)|^2_2+|\nabla\cdot\phi(0)|^2_2\right)\\
&\ge\int_0^T\left(|\sqrt{\r}\phi_t|^2_2+k|\triangle\phi+\nabla(\nabla\cdot\phi)|^2_2-\f{3}{4}|\sqrt{\r}\phi_t|^2_2-\f{k}{3}|\triangle\phi+\nabla(\nabla\cdot\phi)|^2_2\right)dt\\
&\quad+\f{1}{2}\left(|\nabla\phi(T)|^2_2+|\nabla\phi(0)|^2_2+|\nabla\cdot\phi(T)|^2_2+|\nabla\cdot\phi(0)|^2_2\right)\\
&\ge C\|\phi\|_\Phi^2.
\end{split}
\end{equation*}
Hence, by the Lax-Milgram theorem
(see \cite{GT}), there exists a $\u\in \Phi$ such that
\begin{equation}\label{cl6}
E(\u,\phi)=L(\phi)
\end{equation}
for every $\phi\in\Phi$.

If we assume that $\ov{\phi}$ is a solution of the problem
\begin{equation*}
\begin{cases}
&\ov{\phi}_t-k(\triangle\ov{\phi}+\nabla\big(\nabla\cdot\ov{\phi})\big)=0,\\
&\ov{\phi}(0)=h(x)
\end{cases}
\end{equation*}
with $h(x)$ smooth enough,  and
 replace in \eqref{cl6} $\phi$ by
$\ov{\phi}$, then we have
$$\big(\u(0)-\u_0,\  \triangle h+\nabla(\nabla\cdot h)\big)=0,$$ which implies $\u(0)=\u_0$.
Similarly, if $\widetilde{\phi}$ is a solution of the problem
\begin{equation*}
\begin{cases}
& \widetilde{\phi}_t-k\big(\triangle\widetilde{\phi}+\nabla(\nabla\cdot\widetilde{\phi})\big)=g(x,t),\\
&\widetilde{\phi}(0)=0
\end{cases}
\end{equation*}
with $g$ smooth enough, replacing $\phi$ by $\widetilde{\phi}$ in
\eqref{cl6}, then we get
$$\int_0^T\big(\r\u_t-\triangle\u-\nabla(\nabla\cdot\u)+\r\bar{\u}\cdot\nabla
\bar{\u}+\nabla P-(\nabla\times\H)\times\H, \ g\big)\ dt=0,$$ which
implies that $(\u,\r,\H)$ satisfies \eqref{e3} a.e. in $Q_T$.

Next, we prove the higher regularity for $\u$. To avoid tedious
calculations and notation, we work directly with the derivatives
with respect to $t$ of $\u$ instead of its differential quotients.
Multiplying \eqref{e32} by $\u_t$, integrating over $Q_t$ and using
the Cauchy-Schwarz inequality, we obtain
\begin{equation}\label{t}
\begin{split}
&\int_0^t|\sqrt{\r}\u_t|_2^2 \ d\tau+\frac{1}{2}|\nabla\u|_2^2
+\frac{1}{2}|\nabla\cdot\u|_2^2\\&=\frac{1}{2}|\nabla\u(0)|_2^2+\frac{1}{2}|\nabla\cdot\u(0)|_2^2+\int_0^t\big(-\r\bar{\u}\cdot\nabla\bar{\u}-\nabla
P+(\nabla\times\H)\times\H,\ \u_t\big)d\tau\\
&\leq\frac{1}{2}|\nabla\u(0)|_2^2+\frac{1}{2}|\nabla\cdot\u(0)|_2^2+\int_0^t\int_{\R^3}(\frac{1}{3}|\sqrt{\r}\u_t|^2
+\frac{3}{4}|\sqrt{\r}\bar{\u}\cdot\nabla\bar{\u}|^2)\ dxd\tau\\
&\quad +\int_0^t\int_{\R^3}P\nabla\cdot\u_t\
dxd\tau+\int_0^t(\H\cdot\nabla\H,\ \u_t )\
d\tau+\frac{1}{2}\int_0^t\int_{\R^3}|\H|^2\nabla\cdot\u_t\ dxd\tau.
\end{split}
\end{equation}
By H\"{o}lder's inequality and the Gagliardo-Nirenberg-Sobolev
inequalities
\begin{equation}\label{gns}
\begin{split}
|\nabla\bar{\u}|_3\le
C|\nabla\bar{\u}|_2^\frac{1}{2}|\triangle\bar{\u}|_2^\frac{1}{2},
\quad |\bar{\u}|_6\leq C|\nabla\bar{\u}|_2,
\end{split}
\end{equation}
 it is easy to deduce from \eqref{t} that
\begin{equation}\label{cl7}
\begin{split}
&\frac{2}{3}\int_0^t|\sqrt{\r}\u_t|_2^2
d\tau+\frac{1}{2}|\nabla\u|_2^2 +\frac{1}{2}|\nabla\cdot\u|_2^2\\&
\leq\frac{1}{2}|\nabla\u(0)|_2^2
+\frac{1}{2}|\nabla\cdot\u(0)|_2^2+\frac{3}{4}\sup_{0\leq\tau\leq
t}|\r|_\infty\int_0^t|\bar{\u}|_6^2|\nabla\bar{\u}|_3^2 d\tau\\
&\quad+\int_0^t\int_{\R^3}P\nabla\cdot\u_t\ dxd\tau+\int_0^t(\H\cdot\nabla\H,\ \u_t
)\ d\tau+\frac{1}{2}\int_0^t\int_{\R^3}|\H|^2\nabla\cdot\u_t\  dxd\tau\\
&\leq\frac{1}{2}|\nabla\u(0)|_2^2
+\frac{1}{2}|\nabla\cdot\u(0)|_2^2+C\sup_{0\leq\tau\leq
t}|\r|_\infty\int_0^t|\nabla\bar{\u}|_2^3|\triangle\bar{\u}|_2\
d\tau\\
&\quad +C_\delta\left(\int_0^t|P|_2^2\ d\tau+\int_0^t|\frac{1}{\sqrt{\r}}\H|_\infty^2|\nabla\H|_2^2\
d\tau+\int_0^t|\H|_\infty^2|\H|_2^2 d\tau\right)\\
&\quad+\delta\int_0^t|\sqrt{\r}\u_t|_2^2\ d\tau
+2\delta\int_0^t|\nabla\cdot\u_t|_2^2\ d\tau.
\end{split}
\end{equation}

Now we differentiate \eqref{e32} with respect to $t$ and get
\begin{equation}\label{cl8}
\begin{split}
&\r_t\u_t+\r\u_{tt}-\triangle\u_t-\nabla(\nabla\cdot\u_t)\\
&=-\r_t \bar{\u}\cdot\nabla\bar{\u}
-\r\bar{\u}_t\cdot\nabla\bar{\u}-\r
\bar{\u}\cdot\nabla\bar{\u}_t-\nabla P_t+\H_t\cdot\nabla\H
+\H\cdot\nabla\H_t-\nabla(\H\cdot\H_t).
\end{split}
\end{equation}
Multiplying \eqref{cl8} by $\u_t$, integrating over $\R^3$, bearing
in mind the continuity equation \eqref{e31} and the induction
equation \eqref{e33}, we obtain

\begin{equation}\label{cl9}
\begin{split}
&\f{1}{2}\f{d}{dt}|\sqrt{\r}\u_t|^2_2+\f{1}{2}\int_{\R^3}\r_t|\u_t|^2dx+|\nabla\u_t|^2_2+
|\nabla\cdot\u_t|^2_2\\&\le|\r|_\infty|\bar{\u}|_\infty|\nabla
\bar{\u}|_2|\nabla
\bar{\u}|_3|\u_t|_6+|\r|_\infty|\bar{\u}|_6^2|\nabla(\nabla\bar{\u})|_2|\u_t|_6+|\r|_\infty|\bar{\u}|_\infty|\bar{\u}|_6|\nabla
\bar{\u}|_3|\nabla\u_t|_2\\
&\quad+|\sqrt{\r}|_\infty|\sqrt{\r}\bar{\u}_t |_2|\nabla
\bar{\u}|_3|\u_t|_6+|\sqrt{\r}|_\infty|\nabla\bar{\u}_t
|_2|\bar{\u}|_\infty|\sqrt{\r}\u_t|_2+\|P'\|_{C^0_{loc}}|\r_t
|_2|\nabla\u_t|_2\\
&\quad+\big(\nabla\times (\bar{\u}\times \H)\cdot\nabla\H, \u_t\big)+\big(\H\cdot\nabla\H_t ,\ \u_t\big)+|\H|_\infty|\H_t|_2|\nabla\cdot\u_t|_2\\
&=|\r|_\infty|\bar{\u}|_\infty|\nabla \bar{\u}|_2|\nabla
\bar{\u}|_3|\u_t|_6+|\r|_\infty|\bar{\u}|_6^2|\nabla(\nabla\bar{\u})|_2|\u_t|_6+|\r|_\infty|\bar{\u}|_\infty|\bar{\u}|_6|\nabla
\bar{\u}|_3|\nabla\u_t|_2\\
&\quad+|\sqrt{\r}|_\infty|\sqrt{\r}\bar{\u}_t |_2|\nabla
\bar{\u}|_3|\u_t|_6+|\sqrt{\r}|_\infty|\nabla\bar{\u}_t
|_2|\bar{\u}|_\infty|\sqrt{\r}\u_t|_2+\|P'\|_{C^0_{loc}}|\r_t
|_2|\nabla\u_t|_2\\
&\quad-\int_{\R^3}\H\cdot(\nabla\u_t)\nabla\times (\bar{\u}\times
\H)\ dx-\int_{\R^3}\H_t\cdot(\nabla\u_t)\ \H\
dx+|\H|_\infty|\H_t|_2|\nabla\cdot\u_t|_2
\\
&\leq |\r|_\infty|\bar{\u}|_\infty|\nabla \bar{\u}|_2|\nabla
\bar{\u}|_3|\u_t|_6+|\r|_\infty|\bar{\u}|_6^2|\nabla(\nabla\bar{\u})|_2|\u_t|_6+|\r|_\infty|\bar{\u}|_\infty|\bar{\u}|_6|\nabla
\bar{\u}|_3|\nabla\u_t|_2\\
&\quad+|\sqrt{\r}|_\infty|\sqrt{\r}\bar{\u}_t |_2|\nabla
\bar{\u}|_3|\u_t|_6+|\sqrt{\r}|_\infty|\nabla\bar{\u}_t
|_2|\bar{\u}|_\infty|\sqrt{\r}\u_t|_2+\|P'\|_{C^0_{loc}}|\r_t
|_2|\nabla\u_t|_2\\
&\quad+|\H|_\infty|\bar{\u}|_\infty|\nabla\H|_2|\nabla\u_t|_2+|\H|_\infty^2|\nabla\bar{\u}|_2|\nabla\u_t|_2
+|\H|_\infty^2|\nabla\cdot\bar{\u}|_2|\nabla\u_t|_2\\
&\quad+|\H|_\infty|\H_t|_2|\nabla\u_t|_2+|\H|_\infty|\H_t|_2|\nabla\cdot\u_t|_2.
\end{split}
\end{equation}
Integrating \eqref{cl9} with respect to $t$, taking advantage of the
continuity equation and the Gagliardo-Nirenberg-Sobolev inequalities
as \eqref{gns}, we find, since $\r\in L^\infty(Q_T),$
\begin{equation}\label{c20}
\begin{split}
&\f{1}{2}|\sqrt{\r}\u_t|^2_2+\int_0^t(|\nabla\u_t|_2^2+|\nabla\cdot\u_t|^2_2)d\tau\\&\le
\frac{1}{2}|\sqrt{\r(0)}\u_t(0)|^2_2+C\int_0^t|\r|_\infty(|\bar{\u}|_\infty|\nabla\bar{\u}|_2^{\frac{3}{2}}
|\triangle\bar{\u}|_2^\frac{1}{2}|\nabla\u_t|_2+|\nabla\bar{\u}|_2^2|\triangle\bar{\u}|_2|\nabla\u_t|_2)d\tau\\
&\quad+C\int_0^t|\sqrt{\r}|_\infty(|\sqrt{\r}\bar{\u}_t|_2|\nabla\bar{\u}|_2^{\frac{1}{2}}|\triangle\bar{\u}|_2^{\frac{1}{2}}|\nabla\u_t|_2
+|\bar{\u}|_\infty|\nabla\bar{\u}_t|_2|\sqrt{\r}{\u}_t|_2)d\tau\\
&\quad+\int_0^t\|P'\|_{C^0_{loc}}|\r_t
|_2|\nabla\u_t|_2d\tau+\int_0^t\big(|\H|_\infty|\bar{\u}|_\infty|\nabla\H|_2|\nabla\u_t|_2+|\H|_\infty^2|\nabla\bar{\u}|_2|\nabla\u_t|_2
\\
&\quad\quad+|\H|_\infty^2|\nabla\cdot\bar{\u}|_2|\nabla\u_t|_2+|\H|_\infty|\H_t|_2|\nabla\u_t|_2+|\H|_\infty|\H_t|_2|\nabla\cdot\u_t|_2\big)d\tau
\\&\le\frac{1}{2}|\sqrt{\r(0)}\u_t(0)|^2_2+C_\delta\big\{\sup_{0\leq\tau\leq t}|\r|_\infty^2\int_0^t(|\bar{\u}|_\infty^2|\nabla\bar{\u}|_2^3|\triangle\bar{\u}|_2
+|\nabla\bar{\u}|_2^4|\triangle\bar{\u}|_2^2)d\tau\\
&\quad+\sup_{0\leq\tau\leq
t}|\sqrt{\r}|_\infty^2\int_0^t|\sqrt{\r}\bar{\u}_t|^2_2|\nabla\bar{\u}|_2|\triangle\bar{\u}|_2\
d\tau +\sup_{0\leq\tau\leq
t}|\sqrt{\r}|_\infty^2\int_0^t|\bar{\u}|_\infty^2|\sqrt{\r}\u_t|^2_2\ d\tau\\
&\quad+\int_0^t\|P'\|^2_{C^0_{loc}}|\r_t
|_2^2\ d\tau+\int_0^t(|\H|_\infty^2|\bar{\u}|_\infty^2|\nabla\H|^2_2+|\H|_\infty^4|\nabla\bar{\u}|_2^2+|\H|_\infty^4|\nabla\cdot\bar{\u}|_2^2\\
&\quad\quad+2|\H|_\infty^2|\H_t|_2^2)\
d\tau\big\}+8\delta\int_0^t|\nabla\u_t|_2^2\
d\tau+\delta\int_0^t|\nabla \cdot\u_t|_2^2\
d\tau+\delta\int_0^t|\nabla\bar{\u}_t|_2^2\ d\tau,
\end{split}
\end{equation}
where $\delta>0$ is small enough.

 Summing \eqref{cl7} and
\eqref{c20}, and for suitable $\delta\ (\delta\leq \frac{1}{16})$,
we can first obtain,  by Gronwall's
inequality,$$\|\sqrt{\r}\u_t\|_{L^\infty_T(L^2(\R^3))}\leq C,$$ and
secondly,
\begin{equation*}
\begin{split}
&\f{1}{2}|\sqrt{\r}\u_t|^2_2+\f{1}{2}|\nabla\u|^2_2+\f{1}{2}|\nabla\cdot\u|_2^2+\f{1}{3}\int_0^t|\sqrt{\r}\u_t|^2_2d\tau+\f{1}{4}\int_0^t(|\nabla\u_t|^2_2
+|\nabla\cdot\u_t|^2_2)d\tau\\
&\le
\frac{1}{2}|\sqrt{\r(0)}\u_t(0)|^2_2+\frac{1}{2}|\nabla\u(0)|_2^2
+\frac{1}{2}|\nabla\cdot\u(0)|_2^2+C_\delta\int_0^t|P|^2_2d\tau
+C_\delta\int_0^t|\frac{1}{\sqrt{\r}}\H|_\infty^2|\nabla\H|_2^2d\tau\\
&\quad+C\sup_{0\leq\tau\leq
t}|\r|_\infty\int_0^t|\nabla\bar{\u}|_2^3|\triangle\bar{\u}|_2 d\tau
+C_\delta\big\{\sup_{0\leq\tau\leq
t}|\r|_\infty^2\int_0^t(|\bar{\u}|_\infty^2|\nabla\bar{\u}|_2^3|\triangle\bar{\u}|_2
+|\nabla\bar{\u}|_2^4|\triangle\bar{\u}|_2^2)d\tau\\
&\quad+\sup_{0\leq\tau\leq
t}|\sqrt{\r}|_\infty^2(\int_0^t|\sqrt{\r}\bar{\u}_t|^2_2|\nabla\bar{\u}|_2|\triangle\bar{\u}|_2d\tau
+\int_0^t|\bar{\u}|_\infty^2|\sqrt{\r}\u_t|^2_2d\tau)+\int_0^t\|P'\|^2_{C^0_{loc}}|\r_t
|_2^2d\tau\\
&\quad+\int_0^t(|\H|_\infty^2|\bar{\u}|_\infty^2|\nabla\H|^2_2+|\H|_\infty^4|\nabla\bar{\u}|_2^2+|\H|_\infty^4|\nabla\cdot\bar{\u}|_2^2+2|\H|_\infty^2|\H_t|_2^2+|\H|_\infty^2|\H|_2^2)d\tau\big\}\\
&\quad+\delta\int_0^t|\nabla\bar{\u}_t|_2^2d\tau,
\end{split}
\end{equation*}
which implies
\begin{equation}\label{cl11}
\begin{cases}
&\sqrt{\r}\u_t\in \big(L^\infty(0,T; L^2(\R^3))\big)^3;\quad
\u_t\in \big(L^2(0,T; H^1(\R^3))\big)^3;\\
&\nabla\u\in \big(L^\infty(0,T; L^2(\R^3))\big)^9.
\end{cases}
\end{equation}

On the other hand, rewrite \eqref{e32} as
\begin{equation}\label{ee}
-\triangle\u-\nabla(\nabla\cdot\u)=-\r\u_t-\r\bar{\u}\cdot\nabla\bar{\u}-\nabla
P+\H\cdot\nabla\H-\frac{1}{2}\nabla(|\H|^2)\equiv {\bf v},
\end{equation}
which is a system with a strongly elliptic left-hand side.

First, by a classical result on elliptic systems, there exists a
positive constant $C$ such that$$\|\u\|_{2,q}\leq
C|\triangle\u+\nabla(\nabla\cdot\u)|_q.$$
Then  ${\bf v}\in \big(L^\infty(0,T;L^2(\R^3))\big)^3$ by \eqref{t} and \eqref{cl11},
thus $\u\in \big(L^\infty(0,T;  H^2(\R^3))\big)^3$.
Second, since $\r$ is bounded from below and $\sqrt{\r}\u_t\in
\big(L^2(Q_T)\big)^3$, we know that $\u_t\in \big(L^2(Q_T)\big)^3$.
Hence, by the Gagliardo-Nirenberg-Sobolev inequality, we get
$$\|\u_t\|_{L^2_T( L^q(\R^3))}\le
C\|\u_t\|_{L^2(Q_T)}^\theta\|\nabla\u_t\|_{L^2(Q_T)}^{1-\theta}\ \
\textrm{as}\  q\in (3,6],$$ for some $\theta\in [0,1)$. This implies
that, by \eqref{cl11}, $\u_t\in \big(L^2(0,T; L^q(\R^3))\big)^3$,
then ${\bf v}\in \big(L^2(0,T;L^q(\R^3))\big)^3.$ A similar
consideration  leads to $\u\in
\big(L^2(0,T;W^{2,q}(\R^3))\big)^3$. Hence, we can conclude that
$\u\in\Psi$.

\subsection{Existence for \eqref{e1}.}

The above argument guarantees the existence and uniqueness of the
solution to the system \eqref{e3} which enable us to define the map
$\u=G(\bar{\u})$ given by the composition of $$f:\ \bar{\u}\rightarrow
\r(\bar{\u}),\quad g:\ \bar{\u}\rightarrow \H(\bar{\u}),\quad d:\
\big(\r(\bar{\u}), \bar{\u}, \H(\bar{\u})\big)\rightarrow \u.$$
Obviously,  the fixed point of $G: \Psi\rightarrow \Psi$ is
the solution of the system \eqref{e1}.
To find a fixed point of  $G$,  we will use the Schauder-Tychonoff fixed point theorem
(Theorem 5.28, \cite{RW}).

Consider the set
\begin{equation*}
\begin{split}
M=\big\{\phi(x,t)\in\R^3\mid & \; \max\big(\|\phi\|_{L^2_T(
W^{2,q}(\R^3)\cap H^2(\R^3))}, \,
 \|\sqrt{\r}\phi_t\|_{L^\infty_T(L^2(\R^3))}, \\&\qquad\quad \|\phi\|_{L^\infty_T( H^2(\R^3))}, \,
\|\phi_t\|_{L^2_T(H^1(\R^3))}\big)\le r\big\}
\end{split}
\end{equation*}
with
$$r^2=\sigma(|\triangle\u_0|_2^2+|\u_0|_\infty^2 |\nabla \u_0|_2^2
+\|\r_0\|_{W^{1,q}\cap H^1}^2 +\|\H_0\|_{W^{1,q}\cap
H^1}^2)/\alpha$$ for some suitable sufficiently large constant
$\sigma\geq 1$, where\ $q\in (3,6]$ .
Clearly, $M$ is a compact set in $\big(L^2(Q_T)\big)^3$. As we are
going to use a fixed point theorem,  we need to show that
$G(M)\subset M$ and $G$ is continuous in $M$ with respect to the
norm in $\big(L^2(Q_T)\big)^3$.

We first prove that $G(M)\subset M$ for some $T=\ov{T}$. Indeed,
assuming $\bar{\u}\in M$, from \eqref{2}, \eqref{20},  and \eqref{y31},
by H\"{o}der's inequality, we know that for $3<q\leq 6$,
\begin{equation}\label{cl12}
\begin{cases}
\alpha\,\textrm{exp}(-Cr\sqrt{t})\le\r(x,t)\le
\textrm{exp}(Cr\sqrt{t}),\\
|\nabla\r|_q\le \left(|\nabla\r_0|_q
+\textrm{exp}\left(Cr\sqrt{t}\right)r\sqrt{t}\right)\textrm{exp}\left(Cr\sqrt{t}\right),\\
|\nabla\H|_q\le \left(|\nabla\H_0|_q
+Cr\sqrt{t}\right)\textrm{exp}\left(Cr\sqrt{t}\right).
\end{cases}
\end{equation}
Hence, from \eqref{c20} and \eqref{cl12}, it follows that
\begin{equation}\label{cl13}
\begin{split}
&\f{1}{2}|\sqrt{\r}\u_t|_2^2+\f{1}{2}\int_0^t|\nabla\u_t|_2^2d\tau \\
&\le \f{1}{2}|\sqrt{\r(0)}\u_t(0)|_2^2
+C_{\dl}t(r^2+\dl)\|\sqrt{\r}\u_t\|_{L^\infty_t(
L^2(\R^3))}^2+C_{\dl}t(r^4+r^6)+Ctr^2.
\end{split}
\end{equation}
Since if we multiply \eqref{e32} by $\u_t$, integrate over $\R^3$
and let $t=0$, then
\begin{equation*}
\begin{split}
&|\sqrt{\r(0)}\u_t(0)|_2^2\\&=\int_{\R^3}\big(\triangle\u_0+\nabla(\nabla\cdot\u_0)-\r_0\u_0\cdot\nabla\u_0-\nabla
P(\r_0)+\H_0\cdot\nabla\H_0-\frac{1}{2}\nabla(|\H_0|^2)\big)\cdot\u_t(0)dx\\
&\le\frac{1}{\sqrt{\alpha}}\big(|\triangle\u_0|_2+|\nabla(\nabla\cdot\u_0)|_2+|\r_0|_\infty|\u_0|_\infty|\nabla\u_0|_2\\
&\quad\quad\quad\quad\quad\quad\quad+\|P'\|_{C^0_{loc}}|\nabla\r_0|_2+2|\H_0|_\infty|\nabla\H_0|_2\big)|\sqrt{\r(0)}\u_t(0)|_2,
\end{split}
\end{equation*}
i.e.,
\begin{equation}\label{r0}
\begin{split}
|\sqrt{\r(0)}\u_t(0)|_2\leq &
\frac{1}{\sqrt{\alpha}}\big(|\triangle\u_0|_2+|\nabla(\nabla\cdot\u_0)|_2+|\r_0|_\infty|\u_0|_\infty|\nabla\u_0|_2\\
&\quad\quad\qquad\qquad+\|P'\|_{C^0_{loc}}|\nabla\r_0|_2+2|\H_0|_\infty|\nabla\H_0|_2\big).
\end{split}
\end{equation}
Combining \eqref{cl13} and \eqref{r0}, taking $\delta$ and $\ov{T}$
suitably small, we derive that
\begin{equation*}
\|\sqrt{\r}\u_t\|_{L^\infty_{\ov{T}}\big(
L^2(\R^3)\big)}^2+\|\nabla\u_t\|_{L^2(Q_{\ov{T}})}^2\le \f{1}{3}r^2,
\end{equation*}
\begin{equation}\label{cl14}
\|\sqrt{\r}\u_t\|_{L^\infty_{\ov{T}}(
L^2(\R^3))}^2+\|\u_t\|_{L^2_{\ov{T}}(H^1(\R^3))}^2\le \f{2}{3}r^2.
\end{equation}
Then, we estimate the norm $\|\u\|_{L_{\ov{T}}^\infty(
H^2(\R^3))}$ and the norm $\|\u\|_{L^2_{\ov{T}}( W^{2,q}(\R^3))}$.
Indeed, taking advantage of \eqref{ee}, on the one hand, we get
\begin{equation*}
\begin{split}
&\|\triangle\u+\nabla(\nabla\cdot\u)\|_{L^\infty_{\ov{T}}(L^2(\R^3))}\\&\le
\|\sqrt{\r}\|_{L^\infty(Q_{\ov{T}})}\|\sqrt{\r}\u_t\|_{L^\infty_{\ov{T}}(L^2(\R^3))}
+\|\r\|_{L^\infty(Q_{\ov{T}})}\|\bar{\u}\|_{L^\infty(Q_{\ov{T}})}\|\nabla
\bar{\u}\|_{L^\infty_{\ov{T}}(L^2(\R^3))}\\&\quad+C\|\nabla\r\|_{L^\infty_{\ov{T}}(
L^2(\R^3))}+C\|\nabla\H\|_{L^\infty_{\ov{T}}(L^2(\R^3))},
\end{split}
\end{equation*}
which leads to
$$\|\u\|^2_{L^\infty_{\ov{T}}(H^2(\R^3))}\le r^2$$
from a classical result for elliptic systems.
On the other hand, we have, also by the classical result on elliptic
systems,
\begin{equation*}
\begin{split}
\int_0^{\ov{T}}\|\u\|_{2,q}^2\ dt&\le
C\int_0^{\ov{T}}(|\bar{\u}|^2_{L^\infty}|\nabla
\bar{\u}|_q^2+|\nabla\r|_q^2+|\u_t|_q^2+|\H|_\infty^2|\nabla\H|_q^2)dt\\
&\le Cr^4\ov{T}+r^2\ov{T}\\&\le r^2
\end{split}
\end{equation*}
for some suitable small $\ov{T}$. Hence, we have shown that $G(M)\subset
M$.

Next, we prove the continuity of $G$ in $M$. We observe that if
$\{\bar{\u}_n\}_{n=1}^\infty\subset M$, then there exists a
subsequence (still denoted by $\{\bar{\u}_n\}_{n=1}^\infty$) such
that $$\bar{\u}_n\rightarrow \bar{\u} \ \  \textrm{strongly in}\ M\
\textrm{as}\ \ n\rightarrow \infty.$$
 Let $\r_n$ and $\r$ be the
solutions of
$${\r_n}_t+\nabla\cdot(\r_n\bar{\u}_n)=0,\quad \r_n(0)=\r_0,$$ and $$\quad \r_t+\nabla\cdot(\r\bar{\u})=0,\quad \r(0)=\r_0,$$
respectively.
 Denote $\ov{\r}_n=\r_n-\r$, then $\ov{\r}_n$ satisfies
\begin{equation}\label{c}
{\ov{\r}_n}_t+\bar{\u}_n\cdot\nabla\ov{\r}_n+(\bar{\u}_n-\bar{\u})\cdot\nabla\r+\ov{\r}_n\nabla\cdot\bar{\u}_n
+\r\nabla\cdot(\bar{\u}_n-\bar{\u})=0, \quad \ov{\r}_n(0)=0.
\end{equation}
 Multiplying \eqref{c} by ${\ov{\r}_n}$,
integrating over $Q_T$, and applying Gronwall's inequality, it is
easy to derive that
$$|\ov{\r}_n|_2^2\le
\textrm{exp}(Cr\ov{T})\int_0^{\ov{T}}\left(|(\bar{\u}-\bar{\u}_n)\cdot\nabla\r|_2^2+|\r\nabla\cdot(\bar{\u}_n-\bar{\u})|_2^2\right)dt,$$
which implies that $\r_n\rightarrow\r$ strongly in
$L^\infty\big(0,\ov{T}; L^2(\R^3)\big)$.

Similarly, we can show that $\H_n\rightarrow\H$ strongly in
$\big(L^\infty(0,\ov{T}; L^2(\R^3))\big)^3$.

 Now let $\u_n$ and $\u$ be
the solutions of \eqref{e32} corresponding to $\bar{\u}_n$ and
$\bar{\u}$ with $$\u_n(0)=\u(0)=\u_0.$$ Then we have, denoting ${\bf
U}_n=\u_n-\u$ and ${\bf\bar{U}}_n=\bar{\u}_n-\bar{\u}$,

\begin{equation}\label{cl31}
\begin{split}
&\r_n {{\bf U}_n}_t-\triangle {\bf U}_n-\nabla\nabla\cdot {\bf U}_n \\
&=-\ov{\r}_n\u_t-\r_n {\bf\bar{U}}_n\cdot\nabla \bar{\u}_n-\ov{\r}_n\bar{\u}\cdot\nabla \bar{\u}_n-\r \bar{\u}\cdot\nabla {\bf\bar{U}}_n\\
&\quad+\H_n\cdot\nabla\H_n-\H\cdot\nabla\H-\nabla P(\r_n)+\nabla
P(\r).
\end{split}
\end{equation}
Multiplying \eqref{cl31} by ${{\bf U}_n}_t$, integrating over
$Q_{\ov{T}}$, and take advantage of the convergence of $\r_n$ and
$\H_n$, we can prove as a routine matter that $$\nabla
{\bf\bar{U}}_n\rightarrow 0 \ \textrm{strongly in}\
\big(L^2(Q_{\ov{T}})\big)^9$$ and $$\quad \sqrt{\r_n}{{\bf U}_n}_t
\rightarrow 0 \ \textrm{strongly in}\ \big(L^2(Q_{\ov{T}})\big)^3.$$
Due to the convergence of $\r_n$, we deduce that $${{\bf
U}_n}_t\rightarrow 0 \ \textrm{strongly in}\
\big(L^2(Q_{\ov{T}})\big)^3.$$ Hence, by using the identity ${{\bf
U}_n}(t)=\int_0^t{{\bf U}_n}_td\tau \ \big({\bf U}_n(0)=0\big)$, we
get $${\bf U}_n\rightarrow 0  \ \textrm{ strongly in}\
\big(L^2(Q_{\ov{T}})\big)^3.$$ Thus, the map $G$ is continuous in
$M$. The existence of a local solution is completely proved.

\bigskip\bigskip
\section{Uniqueness}
We proceed to prove the uniqueness of the solution by the same
procedure as that used for the continuity of $G$.
We have already proved that $\textrm{for} \ 3<q\leq 6$,
\begin{equation*}
\begin{split}
&\u_t \in \big(L^2(0,\ov{T}; L^2(\R^3)\cap L^q(\R^3))\big)^3, \\
&\nabla \r\in
\big(L^2(0,\ov{T}; L^2(\R^3)\cap L^q(\R^3))\big)^3, \; \\
&\nabla \H\in \big(L^2(0,\ov{T};
L^2(\R^3)\cap L^q(\R^3))\big)^9.
\end{split}
\end{equation*}
Using the standard interpolation, we get
\begin{equation*}
\begin{split}
&\u_t\in \big(L^{2}(0,\ov{T}; L^3(\R^3))\big)^3,\\
& \nabla \r\in \big(L^2(0,\ov{T}; L^3(\R^3))\big)^3,\\
&\nabla \H\in \big(L^2(0,\ov{T}; L^3(\R^3))\big)^9,
 \end{split}
\end{equation*}
 where $$\f{1}{3}=\f{\theta}{2}+\f{1-\theta}{q}.$$

 Now, assume that $\u_1$, $\u_2$ satisfy \eqref{e1}
for some $T>0$ and denote $$\r:=\r(\u_1)-\r(\u_2)=\r_1-\r_2,\
\u:=\u_1-\u_2, \  \H:=\H(\u_1)-\H(\u_2)=\H_1-\H_2.$$ Then, we have
\begin{equation}\label{12}
\r_t +\nabla\r\cdot\u_1+\nabla
\r_2\cdot\u+\r\nabla\cdot\u_1+\r_2\nabla\cdot\u=0,
\end{equation}
with $\r(0)=0.$
 Multiplying \eqref{12} by $\r$ and integrating over
$\R^3$, we have
\begin{equation*}
\begin{split}
&\f{1}{2}\f{d}{dt}|\r|_2^2-\f{1}{2}\int_{\R^3}|\r|^2\nabla\cdot
\u_1dx+\int_{\R^3}\r\nabla \r_2\cdot\u dx
+\int_{\R^3}|\r|^2\nabla\cdot \u_1dx+\int_{\R^3}\r\r_2\nabla\cdot\u
dx=0.
\end{split}
\end{equation*}
Combining  the Cauchy-Schwarz inequality,  the H\"{o}lder inequality and $|\u|_6\leq
C|\nabla\u|_2,$
 we get,
\begin{equation}\label{1213}
\begin{split}
\f{d}{dt}|\r|^2_2&\le |\nabla\cdot\u_1|_\infty|\r|^2_2+2|\u|_6
|\r\nabla \rho_2|_{\f{6}{5}}+\varepsilon|\nabla\u|^2_2
+C_\varepsilon|\rho_2|_\infty^2|\r|^2_2\\&\le
|\nabla\cdot\u_1|_\infty|\r|^2_2+\varepsilon|\nabla\u|_2^2
+C_\varepsilon|\r\nabla
\rho_2|^2_{\f{6}{5}}+\varepsilon|\nabla\u|^2_2
+C_\varepsilon|\rho_2|_\infty^2|\r|^2_2\\
&\le
|\nabla\cdot\u_1|_\infty|\r|^2_2+\varepsilon|\nabla\u|_2^2+C_\varepsilon|\nabla
\rho_2|^2_3|\r|_2^2+\varepsilon|\nabla\u|^2_2
+C_\varepsilon|\rho_2|_\infty^2|\r|^2_2\\
&\le\eta_1(\varepsilon)|\r|_2^2+2\varepsilon|\nabla\u|^2_2,
\end{split}
\end{equation}
where $\eta_1(\varepsilon)=|\nabla\cdot
\u_1|_\infty+C_\varepsilon(|\nabla \rho_2|^2_3+|\rho_2|^2_\infty),\
\varepsilon>0.$

 Similarly, we have
\begin{equation}\label{121}
\H_t
+\u_1\cdot\nabla\H+\u\cdot\nabla\H_2=\big(\nabla\u_1-(\nabla\cdot\u_1){\bf
I}\big)\ \H+\big(\nabla\u-(\nabla\cdot\u){\bf I}\big)\ \H_2,
\end{equation}
with $\H(0)=0$.
Using the same technique as for $\rho$, we get
\begin{equation*}
\begin{split}
&\f{1}{2}\f{d}{dt}|\H|_2^2-\f{1}{2}\int_{\R^3}|\H|^2\nabla\cdot
\u_1\ dx+\int_{\R^3}\H\cdot(\u\cdot\nabla\H_2)\
dx\\&=\int_{\R^3}\H\cdot\big(\nabla\u_1-(\nabla\cdot\u_1){\bf
I}\big)\ \H \ dx+\int_{\R^3}\H\cdot\big(\nabla\u-(\nabla\cdot\u){\bf
I}\big)\ \H_2\ dx,
\end{split}
\end{equation*}
and
\begin{equation}\label{1212}
\begin{split}
\f{d}{dt}|\H|^2_2&\le C|\nabla\cdot
\u_1|_\infty|\H|^2_2+\varepsilon|\nabla \u|_2^2
+C_\varepsilon|\H_2|_\infty^2|\H|_2^2+\varepsilon|\nabla
\u|_2^2+C_\varepsilon|\H\cdot\nabla\H_2|_{\frac{6}{5}}^2\\
&\le C|\nabla\cdot\u_1|_\infty|\H|^2_2+2\varepsilon|\nabla
\u|_2^2+C_\varepsilon|\H_2|^2_\infty|\H|^2_2
+C_\varepsilon|\nabla\H_2|_3^2|\H|^2_2\\
&\le\eta_2(\varepsilon)|\H|_2^2+2\varepsilon|\nabla \u|^2_2,
\end{split}
\end{equation}
where $\eta_2(\varepsilon)=C|\nabla\cdot
\u_1|_\infty+C_\varepsilon(|\H_2|^2_\infty+|\nabla \H_2|^2_3),\
\varepsilon>0.$

 For $\u_i$,\ $i=1, 2$,
\begin{equation*}
\begin{cases}
&\r_i{\u_i}_t-\triangle\u_i-\nabla\cdot\u_i=-\r_i
\u_i\cdot\nabla\u_i-\nabla P(\r_i)+\H_i\cdot\nabla\H_i-\frac{1}{2}\nabla(|\H_i|^2),\\
&\u_i(0)=\u_0.
\end{cases}
\end{equation*}
It is easy to derive
\begin{equation}\label{13}
\begin{split}
&\r_1\u_t-\triangle\u-\nabla(\nabla\cdot\u)\\
&=-\r{\u_2}_t
-\r_1\u\cdot\nabla\u_1-\r\u_2\cdot\nabla\u_1-\r_2\u_2\cdot\nabla\u-\nabla
P(\r_1)+\nabla P(\r_2)\\
&\quad+\H_1\cdot\nabla\H_1-\H_2\cdot\nabla\H_2-\frac{1}{2}\nabla(|\H_1|^2)+\frac{1}{2}\nabla(|\H_2|^2),
\end{split}
\end{equation}
with $\u(0)=0.$

Using once more the same technique as for $\rho, \H$, and bearing in
mind the continuity equation, we deduce that
\begin{equation}\label{14}
\begin{split}
&\frac{1}{2}\frac{d}{dt}|\sqrt{\r}_1\u|_2^2+|\nabla\u|^2_2+|\nabla\cdot\u|_2^2\\
&=-\frac{1}{2}\int_{\R^3}|\u|^2\nabla\cdot(\r_1\u_1)dx-\int_{\R^3}
\Big(\r{\u_2}_t\cdot\u+\r_1\u\cdot\nabla\u_1\cdot\u+\r\u_2\cdot\nabla\u_1\cdot\u\\&\quad+\r_2\u_2\cdot\nabla\u\cdot\u+\nabla
P(\r_1)\cdot\u-\nabla
P(\r_2)\cdot\u-\H\cdot\nabla\H_1\cdot\u-\H_2\cdot\nabla\H\cdot\u\\&\quad+\frac{1}{2}\big(\nabla(|\H_1|^2)-\nabla(|\H_2|^2)\big)\cdot\u\Big)dx\\
&=\int_{\R^3}\nabla\u\cdot\u\cdot(\r_1\u_1)dx+\int_{\R^3}
\Big(\r{\u_2}_t\cdot\u-\r_1\u\cdot\nabla\u_1\cdot\u-\r\u_2\cdot\nabla\u_1\cdot\u-\r_2\u_2\cdot\nabla\u\cdot\u\\
&\qquad\qquad-\nabla
P(\r_1)\cdot\u+\nabla
P(\r_2)\cdot\u+\H\cdot\nabla\H_1\cdot\u+\H_2\cdot\nabla\H\cdot\u\\
&\qquad\qquad -\nabla\H_1\cdot\H\cdot\u-\nabla\H\cdot\H_2\cdot\u\Big)dx\\
&\le
C_\varepsilon|\r_1|_\infty^2|\u_1|_\infty^2|\u|_2^2+\varepsilon|\nabla\u|_2^2+\varepsilon|\nabla\u|^2_2
+C_\varepsilon|{\u_2}_t|_3^2|\r|_2^2+|\sqrt{\r}_1|_\infty|\nabla\u_1|_\infty|\u|_2^2\\
&\quad+|\u_2|_\infty|\nabla\u_1|_\infty(|\r|_2^2+|\u|^2_2)
+C_\varepsilon|\r_2|_\infty^2|\u_2|_\infty^2|\u|_2^2+\varepsilon|\nabla\u|_2^2+C_\varepsilon\|P'\|^2_{C^0_{loc}}|\rho|_2^2
+\varepsilon|\nabla\u|_2^2\\
&\quad+\varepsilon|\nabla\u|_2^2+C_\varepsilon|\nabla\H_1|_2^2|\H|^2_2+C_\varepsilon(|\H_1|^2_\infty+|\H_2|_\infty^2)|\H|_2^2+\varepsilon|\nabla\cdot\u|^2_2\\
&\leq
5\varepsilon|\nabla\u|_2^2+\varepsilon|\nabla\cdot\u|^2_2+\eta_3(\varepsilon)(|\r|_2^2+|\u|_2^2+|\H|_2^2),
\end{split}
\end{equation}
where
\begin{equation*}
\begin{split}
&\eta_3(\varepsilon)=C_\varepsilon\Big(|\r_1|_\infty^2|\u_1|_\infty^2+|{\u_2}_t|_3^2+|\sqrt{\r}_1|_\infty|\nabla\u_1|_\infty+|\u_2|_\infty|\nabla\u_1|_\infty
+|\r_2|_\infty^2|\u_2|_\infty^2\\
&\quad\quad\quad\quad\quad\quad\quad\quad\quad\quad+\|P'\|^2_{C^0_{loc}}+|\nabla\H_1|_2^2+|\H_1|^2_\infty+|\H_2|_\infty^2\Big).
\end{split}
\end{equation*}

Summing up \eqref{1213},\eqref{1212} and \eqref{14}, by taking
$\varepsilon\leq\frac{1}{14}$, we obtain
\begin{equation}\label{15}
\begin{split}
&\frac{d}{dt}(|\sqrt{\r}_1\u|_2^2+|\r|_2^2+|\H|_2^2)+|\nabla\u|^2_2\leq
2\big(\eta_1(\varepsilon)+\eta_2(\varepsilon)+\eta_3(\varepsilon)\big)(|\u|_2^2+|\r|_2^2+|\H|_2^2)\\
&\quad\quad\quad\quad\quad\quad\quad\quad\quad\quad\quad\quad\quad\quad\quad\leq\eta(\varepsilon,
t)(|\sqrt{\r}_1\u|_2^2+|\r|_2^2+|\H|_2^2),
\end{split}
\end{equation}
where
\begin{equation*}
\eta(\varepsilon,t)=2\big(\eta_1(\varepsilon)+\eta_2(\varepsilon)+\eta_3(\varepsilon)\big)\max\{1,\frac{1}{\alpha}\textrm{exp}(Cr\sqrt{t})
\}.
\end{equation*}
And, the integrability of $\eta(\varepsilon,t)$ with respect to $t$
on $[0,T]$ comes from the regularity of $\u_1, \u_2$ and the
estimates in Lemmas \ref{r} and \ref{l1} for $\r_i$, $\H_i$ with
$i=1,2.$ Hence,
$$|\sqrt{\r}_1\u|_2^2+|\r|_2^2+|\H|_2^2=0\ \ \textrm{for all} \ t\in (0,T)$$
follows from Gronwall's inequality, and consequently,
$$\u\equiv 0,\quad \r\equiv 0,\quad \H\equiv 0 \ \ \textrm{on} \ Q_T.$$
The proof of uniqueness is complete.

\bigskip\bigskip

\section*{Acknowledgments}
X. Li's research was supported in part by
the NSAF of China (China Scholarship Council).
D. Wang's research was supported in part by the National Science
Foundation under grant DMS-0906160, and by the
Office of Naval Research under grant N00014-07-1-0668.

\end{document}